\documentclass[12pt,reqno]{amsart}

\usepackage{times}
\usepackage[T1]{fontenc}
\usepackage{mathrsfs}
\usepackage{latexsym}
\usepackage[dvips]{graphics}
\usepackage[titletoc,title]{appendix}
\usepackage{epsfig}

\usepackage{amsmath,amsfonts,amsthm,amssymb,amscd}
\input amssym.def
\input amssym.tex
\usepackage{color}
\usepackage{hyperref}

\usepackage{color}
\usepackage{breakurl}
\usepackage{comment}
\newcommand{\bburl}[1]{\textcolor{blue}{\url{#1}}}

\newcommand{\be}{\begin{equation}}
\newcommand{\ee}{\end{equation}}
\newcommand{\bea}{\begin{eqnarray}}
\newcommand{\eea}{\end{eqnarray}}

\newtheorem{thm}{Theorem}[section]
\newtheorem{conj}[thm]{Conjecture}

\newtheorem{lem}[thm]{Lemma}

\newtheorem{rek}[thm]{Remark}

\numberwithin{equation}{section}

\newcommand{\Mod}[1]{\ \mathrm{mod}\ #1}

\begin{document}

\title{On Even Perfect Numbers II}

\author{H\`ung Vi\d{\^e}t Chu}
\email{\textcolor{blue}{\href{mailto:chuh19@mail.wlu.edu}{chuh19@mail.wlu.edu}}}
\address{Department of Mathematics, University of Illinois at Urbana-Champaign, Urbana, IL 61820, USA}



\begin{abstract}
Let $k>2$ be a prime such that $2^k-1$ is a Mersenne prime. Let $n = 2^{\alpha-1}p$, where $\alpha>1$ and $p<3\cdot 2^{\alpha-1}-1$ is an odd prime. Continuing the work of Cai et al. and Jiang, we prove that $n\ |\ \sigma_k(n)$ if and only if $n$ is an even perfect number $\neq 2^{k-1}(2^k-1)$. Furthermore, if $n = 2^{\alpha-1}p^{\beta-1}$ for some $\beta>1$, then $n\ |\ \sigma_5(n)$ if and only if $n$ is an even perfect number $\neq 496$.
\end{abstract}

\maketitle

\section{Introduction and Main Results}
For a positve integer $n$, let $\sigma(n)$ be the sum of the positive divisors of $n$. We call $n$ \textit{perfect} if $\sigma(n) = 2n$. Due to the work of Euclid and Euler, it is well-known that an even integer $n$ is perfect if and only if $n = 2^{p-1}(2^p-1)$, where both $p$ and $2^{p}-1$ are both primes. A prime of the form $2^p-1$ is called a \textit{Mersenne} prime. Up to now, fewer than 60 Mersenne primes are known. Two questions are still open: whether there are infinitely many even perfect numbers and whether there exists an odd perfect number, though various progress has been made. For example, Pomerance \cite{P1} showed that an odd perfect number must have at least $7$ distinct prime factors. For related results, see \cite{P2, P3}. 

Meanwhile, mathematicians have generalized the concept of perfect numbers. Pollack and Shevelev \cite{PS} introduced \textit{$k$-near-perfect} numbers. For $k\ge 1$, a $k$-near-perfect number $n$ is the sum of all of its proper divisors with at most $k$ exceptions. A positive integer $n$ is called \textit{near-perfect} if it is the sum of all but exactly one of its proper divisors. Pollack and Shevelev showed how to construct near-perfect numbers and established an upper bound of $x^{5/6+o(1)}$ for the number of near-perfect numbers in $[1,x]$ as $x\rightarrow \infty$. Li and Liao \cite{LL} gave two equivalent conditions of all even near-perfect numbers in the form $2^\alpha p_1p_2$ and $2^{\alpha}p_1^2p_2$, where $\alpha > 0$ and $p_1, p_2$ are distinct primes. In 2013, Ren and Chen \cite{RC} found all near-perfect numbers with two distinct prime factors and continuing the work, Tang et al. \cite{TRL} showed that there is no odd near-perfect number with three distinct prime divisors. From another perspective, Chen \cite{Ch} defined \textit{$k$-deficient-perfect numbers} and determined all odd exactly $2$-deficient-perfect numbers with two distinct prime divisors. For more beautiful results on near-perfect numbers and deficient-perfect numbers, see \cite{TF, TMF}. 

The present paper focuses on another generalization of perfect numbers by connecting an even perfect number $n$ with the divisibility of $\sigma_k(n)$, where $k\ge 1$ and  
$$\sigma_k(n) \ :=\ \sum_{d|n} d^k.$$
In 2006, Luca and Ferdinands proved that for $k\ge 2$, there are infinitely many $n$ such that $n\ |\ \sigma_k(n)$. In 2015, Cai et al. \cite{CCZ} proved the following theorem.

\begin{thm}\label{CCZ0}
Let $n = 2^{\alpha-1}p$, where $\alpha>1$ is an integer and $p$ is an odd prime. If $n\ |\ \sigma_3(n)$, then $n$ is an even perfect number. The coverse is also true for $n\neq 28$. 
\end{thm}
About three years later, Jiang \cite{Ji} improved the theorem as follows. 
\begin{thm}\label{Ji0}
Let $n =2^{\alpha-1}p^{\beta-1}$, where $\alpha, \beta >1$ are integers and $p$ is an odd prime. Then $n\ |\ \sigma_3(n)$ if and only if $n$ is an even perfect number $\neq 28$. 
\end{thm}
These theorems show a beautiful relationship between an even perfect number $n$ and $\sigma_3(n)$. A natural extension is to consider $\sigma_k(n)$ for some other values of $k$. Unfortunately, Theorem \ref{CCZ0} does not hold when $k = 5$ or $7$, for example. A quick computer search gives $\sigma_5(22)\equiv 0\Mod 22$ and $\sigma_7(86)\equiv 0\Mod 86$. However, if we add one more restriction on $p$, the following theorem holds. 
\begin{thm}\label{main0}
Let $k>2$ be a prime such that $2^k-1$ is a Mersenne prime. If $n = 2^{\alpha-1}p$, where $\alpha>1$ and $p<3\cdot 2^{\alpha-1}-1$ is an odd prime. Then $n\ |\ \sigma_k(n)$ if and only if $n$ is an even perfect number $\neq 2^{k-1}(2^k-1)$.  
\end{thm}
Theorem \ref{main0} can be considered a generalization of Theorem \ref{CCZ0} as we have a wider range of $k$ with the new restriction on $p$ as a compensation. Interestingly, we $k = 5$, when can generalize Theorem \ref{main0} the same way as Jiang generalized Theorem \ref{CCZ0}. 

\begin{thm}\label{main1} If $n = 2^{\alpha-1}p^{\beta-1}$, where $\alpha, \beta>1$ and $p<3\cdot 2^{\alpha-1} -1$ is an odd prime. Then $n\ |\ \sigma_5(n)$ if and only if $n$ is an even perfect number $\neq 496$.
\end{thm}

Unfortunately, our method is not applicable to other values of $k$ even though computation supports the following conjecture.  
\begin{conj}
Let $k>2$ be a prime such that $2^k-1$ is a Mersenne prime. If $n = 2^{\alpha-1}p^{\beta-1}$, where $\alpha, \beta > 1$ and $p<3\cdot 2^{\alpha-1}-1$ is an odd prime. Then $n\ |\ \sigma_k(n)$ if and only if $n$ is an even perfect number $\neq 2^{k-1}(2^k-1)$.
\end{conj}
Our paper is structured as follows. Section \ref{prelim1} provides several preliminary results that are used repeatedly throughout the paper, Section \ref{proof1} proves Theorem \ref{main0} and Section \ref{proof2} proves Theorem \ref{main1}. Since the proof of several claims made in Section \ref{proof1} and Section \ref{proof2} are quite technical, we move them to the Appendix for the ease of reading.

\section{Preliminaries}\label{prelim1}
Let $n = 2^{\alpha-1}p^{\beta-1}$, where $\alpha, \beta>1$ are integers and $p < 3\cdot 2^{\alpha-1}-1$ is an odd prime. Let $k>2$ be a prime such that $2^k-1$ is a Mersenne prime. We will stick with these notation throughout the paper. If $n\ |\ \sigma_k(n)$, then 
\begin{align*}
    2^{\alpha-1}p^{\beta-1}\ |\ \sigma_k(2^{\alpha-1})\sigma_k(p^{\beta-1})&\ =\ (1+2^k+\cdots + 2^{(\alpha-1)k})(1+p^k+\cdots + p^{(\beta-1)k})\\
    &\ =\ \frac{2^{\alpha k}-1}{2^k-1}\cdot \frac{p^{\beta k} - 1}{p^k-1}.
\end{align*}
Because $(2, 2^{\alpha k} - 1) = 1$ and $(p, p^{\beta k} - 1) = 1$, it follows that
\begin{align}
    \label{k1}&2^{\alpha - 1} \mbox{ divides } \frac{p^{\beta k} - 1}{p^k- 1}, \mbox{ so }  2^{\alpha} \mbox{ divides } p^{\beta k} - 1,\\
    \label{k2}&p^{\beta - 1} \mbox{ divides } \frac{2^{\alpha k} - 1}{2^k - 1}.
\end{align}
Furthermore, rewrite \eqref{k1} as $$2^{\alpha-1}\ |\ \frac{p^{\beta k} - 1}{p^k-1} \ =\ \frac{(p^{k}-1)(p^{k(\beta-1)}+ p^{k(\beta-2)} + \cdots + 1)}{p^k - 1} \ =\ \sum_{i=0}^{\beta-1} p^{ki}.$$
Since each term is odd and the summation is divisible by $2$, we know that $2\ |\ \beta$. The following lemma is the key ingredient in the proof of Theorem \ref{main0}.
\begin{lem}\label{f}
Let $n = 2^{\alpha-1}(2^k-1)^{\beta-1}$, where $\alpha, \beta>1$ are integers. Then $n\ \not|\ \sigma_k(n)$.
\end{lem}
\begin{proof}
We use proof by contradiction. Suppose $n\ |\ \sigma_k(n)$. By \eqref{k1} and \eqref{k2}, we have
\begin{align}
    \label{kk1}2^\alpha &\ |\ (2^k-1)^{\beta k} -1,\\
    \label{kk2} (2^k-1)^{\beta} &\ |\ (2^{\alpha k}-1)\ =\ (2^k-1)((2^k)^{\alpha-1}+\cdots+1).
\end{align}
Write $\alpha = (2^k-1)^u \alpha_1$ and $\beta = 2^v\beta_1$, where $u\ge 0$, $v\ge 1$ and $(2^k-1, \alpha_1) = (2,\beta_1) = 1$. By Lemma \ref{cando}, $\alpha \le v+k$. 

If $u=0$, we get $\alpha = \alpha_1$. From \eqref{kk2}, $\beta = 1$, which contradicts that $2\ |\ \beta$.

If $u\ge 1$, Remark \ref{appr2} implies that $\beta\le u+2^k-1$. We have
\begin{align*}
   2^{(2^k-1)^u-k}\ \le \ 2^{\alpha-k}\beta_1 \ \le \ 2^v\beta_1\ =\ \beta\ \le \ u+2^k-1
\end{align*}
Since for all $u\ge 1$ and $k\ge 3$, $$ 2^{(2^k-1)^u-k}\ > \  u+2^k-1,$$
we have a contradiction. This finishes our proof. \end{proof}
\section{Proof of Theorem \ref{main0}}\label{proof1}
For the backward implication, we prove that if $n = 2^{\alpha-1}p$ and $n\ |\ \sigma_k(n)$, then $\alpha$ is prime and $p = 2^{\alpha}-1$. By Lemma \ref{f}, $n\neq 2^{k-1}(2^k-1)$. We have
\begin{align*}
    \sigma_k(n) \ =\ \sigma_k(2^{\alpha-1}p)&\ =\ \sigma_k(2^{\alpha-1})\sigma_k(p)\\
    &\ =\ (1+2^k+\cdots + 2^{k(\alpha-1)})(1+p^k)\\
    &\ =\ (1+2^k+\cdots + 2^{k(\alpha-1)})(1+p)\sum_{i=1}^k p^{k-i}(-1)^{i+1}.
\end{align*}
So, $2^{\alpha-1}p\ |\ \sigma_k(n)$ implies that $2^{\alpha-1}\ |\ 1+p$ and $p\ |\ 1+2^k+\cdots + 2^{k(\alpha-1)}$. There exist $k_1, k_2\in\mathbb{N}$ such that $p = k_1 2^{\alpha-1}-1$ and $1+2^k+\cdots + 2^{k(\alpha-1)} = \frac{2^{k\alpha}-1}{2^k-1} = k_2p$. So,
\begin{align}\label{sc1}
    2^{k\alpha} - 1 \ =\ (2^{\alpha}-1)\sum_{i=0}^{k-1}2^{i\alpha} \ =\ k_3(k_1 2^{\alpha-1}-1),
\end{align}
where $k_3 = (2^k-1)k_2$.

Suppose that $k_1 = 1$. Then $p = 2^{\alpha-1}-1$ and \eqref{sc1} implies that either $2^{\alpha-1}-1\ |\ (2^\alpha-1)$ or $2^{\alpha-1}-1 \ |\ \sum_{i=0}^{k-1}2^{i\alpha}$. If the former, we write
$$1 \ =\ 2^\alpha - 1 - 2(2^{\alpha-1}-1)\ \equiv\ 0 \Mod 2^{\alpha-1}-1,$$
which is impossible. If the latter, we let $x_0 = 2^{\alpha}$ to have
\begin{align}\label{sc2}\frac{x_0}{2}-1\ |\ \sum_{i=0}^{k-1}x_0^i.\end{align}
Consider two polynomials $f(x) = \sum_{i=0}^{k-1}x^i$ and $g(x) = \frac{x}{2}-1$. By the division algorithm, we write
$f(x) = g(x)p(x) + q(x)$ for some polynomials $p(x)$ and $q(x)$ with $\deg q(x) < \deg g(x) = 1$. Observe that $p(x)$ has integer coefficients. Since $\deg q(x)<\deg g(x)$, $q(x)$ is a constant polynomial. In particular, 
$$q(x) \ =\ q(2)\ =\ f(2) - g(2)p(2)\ =\ f(2) - 0\cdot p(2)\ =\ \sum_{i=0}^{k-1}2^i\ =\ 2^k-1.$$
So, $f(x_0) = g(x_0)p(x_0) + 2^k-1$ and so, $f(x_0)\equiv 2^k-1\Mod g(x_0)$. By \eqref{sc2}, we know that 
$$2^k-1\equiv 0\Mod 2^{\alpha-1}-1,$$
which implies that $p = 2^{\alpha-1}-1 = 2^k-1$. By Lemma \ref{f}, $n\ \not| \ \sigma_k(n)$, a contradiction. So, $k_1\ge 2$; however, $k_1<3$ by assumption. So, $k_1 = 2$; we have $p = 2^{\alpha}-1$ and $\alpha$ is a prime. Therefore, $n$ is an even perfect number $\neq 2^{k-1}(2^k-1)$.

For the forward implication, write $n = 2^{q-1}(2^q-1)$, where $q\neq k$ and $2^q-1$ are primes. We have
\begin{align*}\sigma_k(n) &\ =\ (1+2^k+2^{2k}+\cdots + 2^{(q-1)k})(1+(2^q-1)^k)\\
&\ =\ \frac{2^{qk}-1}{2^k-1}(1+(2^q-1)^k).
\end{align*}
Clearly, $2^{q-1}$ divides $1+(2^q-1)^k$. It suffices to show that $2^q-1$ divides $\frac{2^{qk}-1}{2^k-1}$. The fact $n\neq 2^{k-1}(2^k-1)$ implies that $2^q-1$ and $2^k-1$ are two distinct primes. So, $(2^q-1, 2^k-1) = 1$. Because $2^q-1\ |\ 2^{qk}-1$, $2^q-1$ divides $\frac{2^{qk}-1}{2^k-1}$. Therefore, $n\ |\ \sigma_k(n)$.
\section{Proof of Theorem \ref{main1}}\label{proof2}
\subsection{Preliminary results}
We provide lemmas that give useful bounds used in the proof of Theorem \ref{main1}.
\begin{lem}\label{v10}
Let $n = 2^{\alpha-1}p^3$, where $\alpha>1$, $p\equiv 3\Mod 4$ and $p<3\cdot 2^{\alpha-1}-1$. Then $n\ \not|\ \sigma_5(n)$. 
\end{lem}
\begin{proof} We prove by contradiction. Suppose that $n\ |\ \sigma_5(n)$. 
We have \begin{align*}\sigma_5(2^{\alpha-1}p^{3}) &\ =\ (1+2^5+\cdots 2^{5(\alpha-1)})(1+p^5+p^{10}+p^{15})\\
&\ =\ (1+2^5+\cdots 2^{5(\alpha-1)})(p^{10}+1)(p+1)(p^4-p^3+p^2-p+1).
\end{align*}
So, \begin{align}2^{\alpha-1} &\ |\ (p^{10}+1)(p+1)\\ \label{v11}p^3&\ |\ 1+2^5+\cdots 2^{5(\alpha-1)}\ =\ \frac{2^{5\alpha}-1}{2^5-1}.\end{align} Because $p^{10}+1\equiv 2\Mod 4$, we know that $2^{\alpha-2}\ |\ p+1$. Hence, $p = k_12^{\alpha-2}-1$ for some $k_1\in \mathbb{N}$. Combining with $p<3\cdot 2^{\alpha-1}-1$, we get $1\le k_1\le 5$. By \eqref{v11}, write $2^{5\alpha}-1 = 31k_2p^3$ for some $k_2\in\mathbb{N}$. Therefore,
\begin{align}\label{v12}
    31k_2(k_12^{\alpha-2}-1)^3 \ =\ (2^{\alpha}-1)(2^{4\alpha}+2^{3\alpha}+2^{2\alpha}+2^\alpha+1).
\end{align}

Suppose that $p$ divides both $2^{\alpha}-1$ and $\sum_{i=0}^4 2^{i\alpha}$. Then $2^{\alpha}\equiv 1\Mod p$ and so, $\sum_{i=0}^4 2^{i\alpha}\equiv 5\Mod p$. Hence, $p = 5$, which contradicts $p\equiv 3\Mod 4$. It must be that either $p^3\ |\ \sum_{i=0}^4 2^{i\alpha}$ or $p^3 \ |\ 2^{\alpha}-1$. We consider two corresponding cases. 

Case 1: $(k_12^{\alpha-2}-1)^3\ |\ 2^{\alpha}-1$. So, $(k_12^{\alpha-2}-1)^3 \le  2^{\alpha}-1$. In order that the inequality is true for some $\alpha\ge 2$, $1\le k_1\le 2$. 
\begin{itemize}
    \item [(i)] $k_1 = 1$. Then $2^{\alpha-2}-1\ |\ 2^\alpha-1$. Because 
    $$3\ =\ (2^\alpha-1)-4(2^{\alpha-2}-1)\ \equiv\ 0\Mod 2^{\alpha-2}-1,$$
    $p = 2^{\alpha-2}-1 = 3$. So, $\alpha = 4$ and $n = 2^33^3$, a contradiction as $2^33^3\ \not |\ \sigma_5(2^33^3)$.
    \item [(ii)] $k_1 = 2$. Then $2^{\alpha-1}-1\ |\ 2^\alpha-1$. Because 
    $$1\ =\ (2^\alpha-1)-2(2^{\alpha-1}-1)\ \equiv\ 0\Mod 2^{\alpha-1}-1,$$
    $p = 2^{\alpha-1}-1 = 1$, a contradiction. 
\end{itemize}

Case 2: $(k_12^{\alpha-2}-1)^3\ |\ \sum_{i=0}^4 2^{i\alpha}$. Let $x_0 = 2^\alpha$. Let $f(x) = x^4+x^3+x^2+x+1$ and $g(x) = k_1x/4-1$. Clearly, $p = g(x_0)\ |\ f(x_0)$. By the division algorithm, we can write $f(x) = p(x)g(x)+q(x)$, where $q(x)$ is a constant polynomial.
\begin{itemize}
\item[(i)] If $k_1 = 1$, we have $p(x) = 4x^3+20x^2+84x+340$ and $q(x) = 341$. So, $f(x_0) = p(x_0)g(x_0)+341$. Take modulo $g(x_0)$ to have $341\equiv 0\Mod g(x_0)$. Hence, $p\ |\ 341$ and so $p = 11$ or $31$. Since $p = 2^{\alpha-2}-1$, $p = 31$ and $\alpha = 7$. However, $n = 2^631^3\ \not|\ \sigma_5(n)$. 

\item[(ii)] If $k_1 = 2$, we have $p(x) = 2x^3+6x^2+14x+10$ and $q(x) = 31$. So, $f(x_0) = p(x_0)g(x_0) + 31$. Take modulo $g(x_0)$ to have $31\equiv 0\Mod g(x_0)$. Hence, $p\ |\ 31$ and so $p = 31$, $\alpha = 6$. However, $n = 2^531^3\ \not|\ \sigma_5(n)$.

\item[(iii)] If $k_1 = 3$, we have $p(x) = \frac{4}{3}x^3+\frac{28}{9}x^2+\frac{148}{27}x+\frac{700}{81}$ and $q(x) = \frac{781}{81}$. So, $81f(x_0) = (108x_0^3+252x_0^2+444x_0+700)g(x_0)+781$. Take modulo $g(x_0)$ to have $781\equiv 0\Mod g(x_0)$. Then $p\ |\ 781$ and so, $p= 3\cdot 2^{\alpha-2}-1 = 11$, $\alpha = 4$ and $n = 2^311^3$. However, $n = 2^311^3\ \not|\ \sigma_5(n)$.

\item[(iv)] If $k_1 = 4$, we have $p(x) = x^3+2x^2+3x+4$ and $q(x) = 5$. So, $f(x_0) = p(x_0)g(x_0) + 5$. Take modulo $g(x_0)$ to have $5\equiv 0\Mod g(x_0)$. So, $p = 5$, which contradicts that $p\equiv 3\Mod 4$. 

\item[(v)] If $k_1 = 5$, we have $p(x) = \frac{4}{5}x^3+\frac{36}{25}x^2+\frac{244}{125}x+\frac{1476}{625}$ and $q(x) = \frac{2101}{625}$. So, $625f(x_0) = (500x_0^3+900x_0^2+1220x_0+1476)g(x_0) + 2101$. Take modulo $g(x_0)$ to have 
$2101\equiv 0\Mod g(x_0)$. So, $p = 5\cdot 2^{\alpha-2}-1 = 11$ or $191$. Both cases are impossible. 
\end{itemize}
This completes our proof. 
\end{proof}
\begin{lem}\label{not1mod4}
Let $n= 2^{\alpha-1}p^{\beta-1}$, $p\equiv 1\Mod 4$ and $n\ |\ \sigma_k(n)$. Write $\beta = 2^v\beta_1$, where $v\ge 1$ and $(2,\beta_1)=1$. Then
\begin{align}\label{u1}p^{2^v-1}\ \le\ \frac{2^{k(v+1)}-1}{2^k-1}.\end{align}
\end{lem}
\begin{proof}
Let $p-1 = 2^tp_1$, where $t\ge 2$ and $2\ \not|\ p_1$. Because \begin{align}\label{sl}p^k-1 = (p-1)\sum_{i=1}^{k}p^{k-i} = 2^tp_1\sum_{i=1}^{k}p^{k-i},\end{align} we have $2^t\ ||\ (p^k-1)$. By Lemma \ref{tv}, $2^{t+v}\ ||\ p^{k\beta}-1$. Hence,
$$2^v \ ||\ \frac{p^{k\beta}-1}{p^k-1}.$$
By \eqref{k1}, \begin{align}\label{u2}\alpha\le v+1.\end{align} and so
\begin{align*}
    p^{2^v-1} \ \le\ p^{\beta-1}\ \le\ \frac{2^{k\alpha}-1}{2^{k}-1}\ \le\ \frac{2^{k(v+1)}-1}{2^k-1}.
\end{align*}
\end{proof}
\begin{lem}\label{not-1mod4}
Let $n= 2^{\alpha-1}p^{\beta-1}$, $p\equiv 3\Mod 4$ and $n\ |\ \sigma_k(n)$. Write $\beta = 2^v\beta_1$, where $v\ge 1$ and $(2,\beta_1)=1$. Then
\begin{align}\label{v3}p^{2^v-2k-1}\ <\ \frac{2^{k(v-1)}}{2^k-1}.\end{align}
\end{lem}
\begin{proof}
Let $p^2-1 = 2^sp_2$, where $2\ \not|\ p_2$. Then $s\ge 3$. By \eqref{sl}, $2\ ||\ p^k-1$ and by Lemma \ref{tv2}, $2^{v+s-1}\ ||\ p^{k\beta}-1$. Hence, 
$$2^{v+s-2}\ ||\ \frac{p^{k\beta}-1}{p^k-1}.$$
By \eqref{k1}, \begin{align}\label{v4}\alpha\ \le\ v+s-1.\end{align} We have
\begin{align*}
    p^{2^v-1}\ \le\ p^{\beta-1}&\ \le\ \frac{2^{k\alpha}-1}{2^k-1}\ \le\ \frac{2^{k(v+s-1)}-1}{2^k-1}\\
    &\ =\ \frac{2^{ks}2^{k(v-1)}-1}{2^k-1} \ <\ \frac{p^{2k}2^{k(v-1)}-1}{2^k-1} \mbox{ because } p^2> 2^s.
\end{align*}
Therefore,
$$p^{2^v-2k-1}\ <\ \frac{2^{k(v-1)}-1/p^{2k}}{2^k-1}\ <\ \frac{2^{k(v-1)}}{2^k-1}.$$
\end{proof}
\begin{lem}\label{3mod4}
Let $n= 2^{\alpha-1}p^{\beta-1}$, $p\equiv 3\Mod 4$ and $n\ | \ \sigma_k(n)$. Write $\beta = 2^v\beta_1$ and $p+1 = 2^\lambda p_1$, where $(2,\beta_1) = (2, p_1)=1$. Then one of the following must hold
\begin{itemize}
    \item [(1)] $$p\ =\ k,$$
    \item [(2)] $$(2^\lambda-1)^{\beta-1}\ \le\ 2^{\lambda + v}-1,$$
    \item [(3)] $$(2^\lambda - 1)^{\beta-1}\ \le \sum_{i=0}^{k-1} 2^{i(\lambda+v)}.$$
\end{itemize}
\end{lem}
\begin{proof}
From \eqref{k1} and \eqref{k2}, we have 
$$2^{\alpha} \ |\ p^{\beta}-1\mbox{ and } p^{\beta-1}\ |\ 2^{k\alpha}-1\ =\ (2^\alpha - 1)\sum_{i=0}^{k-1} 2^{i\alpha}.$$
By Lemma \ref{sl3}, 
$2^{\lambda+v}\ ||\ p^\beta-1$. So, $\alpha\le \lambda + v$. 

Case 1: $p\ |\ 2^{\alpha}-1$ and $p\ |\ \sum_{i=0}^{k-1} 2^{i\alpha}$. The fact that $2^{\alpha}\equiv 1\Mod p$ implies that $\sum_{i=0}^{k-1} 2^{i\alpha}\equiv k\Mod p$. Because $p\ |\ \sum_{i=0}^{k-1} 2^{i\alpha}$ and $k$ is prime, it must be that $p\ =\ k$. This is scenario (1).

Case 2: $p\ |\ 2^{\alpha} - 1$ and $p \ \not| \ \sum_{i=0}^{k-1} 2^{i\alpha}$. So, 
$$2^{\alpha} \ |\ p^{\beta}-1\mbox{ and } p^{\beta-1}\ |\ 2^\alpha - 1.$$ 
We have
$$(2^\lambda - 1)^{\beta-1}\ \le\ p^{\beta-1} \ \le\ 2^{\alpha}-1\ \le\ 2^{\lambda+v}-1.$$
This is scenario (2). 

Case 3: $p\ \not |\ 2^\alpha - 1$ and $p \ |\ \sum_{i=0}^{k-1} 2^{i\alpha}$. So, 
$$2^{\alpha}\ |\ p^{\beta}-1\mbox{ and }p^{\beta-1}\ |\ \sum_{i=0}^{k-1} 2^{i\alpha}.$$
We have
$$(2^\lambda - 1)^{\beta-1}\ \le\ p^{\beta-1}\ \le\ \sum_{i=0}^{k-1} 2^{i\alpha}\ \le\ \sum_{i=0}^{k-1} 2^{i(\lambda+v)}.$$
This is scenario (3). We have finished our proof. 
\end{proof}

\subsection{Proof of Theorem \ref{main1}} 
We now bring together all preliminary results and prove Theorem \ref{main1} by case analysis. 
\begin{proof}
The backward implication follows from Theorem \ref{main0}. We prove the forward implication. Let $n = 2^{\alpha-1}p^{\beta-1}$, where $\alpha, \beta>1$ and $p<3\cdot 2^{\alpha-1} -1$ is an odd prime. Suppose that $n\ |\ \sigma_5(n)$. Computation shows that $n\neq 496$. 

Case 1: $p\equiv 1\Mod 4$. By \eqref{u1}, 
\begin{align}
    \label{v1} 5^{2^v-1}\ \le \ p^{2^v-1}\ \le\ \frac{2^{5(v+1)}-1}{2^5-1},
\end{align}
which only holds if $1\le v\le 2$. 
\begin{itemize}
    \item [(i)] $v = 1$. By \eqref{u2}, $\alpha = 2$ then by \eqref{k2}, $p\ |\ 33$, which contradicts $p\equiv 1\Mod 4$. 
    \item [(ii)] $v = 2$. By \eqref{v1}, $p\le 10$ and so $p = 5$. By \eqref{u2}, $2\le \alpha \le 3$. However, neither value of $\alpha$ satisfies \eqref{k2}.
\end{itemize}

Case 2: $p\equiv 3\Mod 4$. Note that because $k = 5$, we can ignore scenario (1) of Lemma \ref{3mod4}. By \eqref{v3}, 
\begin{align}\label{v5}
    3^{2^v-11}\ \le\ p^{2^v-11}\ <\ \frac{2^{5(v-1)}}{2^5-1}, 
\end{align}
which implies $1\le v\le 4$.\footnote{We use {\tt Desmos|Graphing Calculator} to figure out this range, which can also be proved by the intermediate value theorem.}
\begin{itemize}
    \item [(i)] $v = 4$. By \eqref{v5}, $p = 3$. So, in \eqref{v4}, $s = 3$ and $2\le \alpha\le 6$. If $\alpha\le 5$, \eqref{k2} gives
    \begin{align*}
        3^{15}\ |\ 3^{16\beta_1 - 1}\ \le\ \frac{2^{25}-1}{31}, \mbox{ a contradiction.}
    \end{align*}
    If $\alpha = 6$, \eqref{k2} does not hold. 
    \item [(ii)] $v = 3$. Then $\beta\ge 8$. By Lemma \ref{3mod4}, either $(2^\lambda - 1)^{\beta-1}\le 2^{\lambda + 3} - 1$ or 
    $(2^\lambda - 1)^{\beta-1} \le \sum_{i=0}^4 2^{i(\lambda + 3)}$.
    \begin{itemize}
        \item[(a)] If $(2^\lambda - 1)^{\beta-1}\le 2^{\lambda + 3} - 1$, then $\lambda < 2$ because $\beta\ge 8$, a contradiction. 
       \item [(b)] If $(2^\lambda - 1)^{\beta-1} \le \sum_{i=0}^4 2^{i(\lambda + 3)}$, then $\beta\le 15$ in order that $\lambda \ge 2$. Since $8\ |\ \beta$, we know $\beta = 8$. Plugging $\beta = 8$ into $(2^\lambda - 1)^{\beta-1} \le \sum_{i=0}^4 2^{i(\lambda + 3)}$, we have $2\le \lambda \le 4$ and so $2\le s\le 5$. By \eqref{v4}, $2\le \alpha \le 7$ and by \eqref{k2}, we acquire
       $$p^7\ |\ \frac{2^{5\alpha}-1}{31}\ \le\ \frac{2^{35}-1}{31}.$$
       Hence, $p \in \{3, 7, 11, 19\}$. Computation shows that for each pair $(\alpha, p)$, $\eqref{k2}$ does not hold. 
    \end{itemize}
    \item [(iii)] $v = 2$. Then $4\ | \beta$. By Lemma \ref{3mod4}, either $(2^\lambda - 1)^{\beta-1}\le 2^{\lambda + 2} - 1$ or 
    $(2^\lambda - 1)^{\beta-1} \le \sum_{i=0}^4 2^{i(\lambda + 2)}$. Since $\beta\ge 4$ and $\lambda \ge 2$, the former does not hold. If the later, since $\lambda \ge 2$, it must be that $\beta<12$ and so $\beta\in\{4, 8\}$. 
    \begin{itemize}
        \item[(a)] $\beta = 4$. Lemma \ref{v10} rejects this case. 
        \item[(b)] $\beta = 8$. Plugging $\beta = 8$ into $(2^\lambda - 1)^{\beta-1} \le \sum_{i=0}^4 2^{i(\lambda + 2)}$, we have $2\le \lambda\le 3$ and so $2\le s\le 4$. By \eqref{v4}, $2\le \alpha\le 5$. This is back to item (ii) part (b). 
    \end{itemize}
    \item [(iv)] $v=1$. By Lemma \ref{3mod4}, either $(2^\lambda - 1)^{\beta-1}\le 2^{\lambda+1}-1$ or $(2^\lambda - 1)^{\beta-1}\le \sum_{i=0}^4 2^{i(\lambda + 1)}$. If the former, $\beta = 2$ and $n = 2^{\alpha-1}p$. By Theorem \ref{main0}, $n$ is an even perfect number. If the latter, since $\lambda \ge 2$, it must be that $\beta\le 9$ and so $\beta \in \{2, 6\}$. 
    \begin{itemize}
    \item[(a)] If $\beta = 2$, Theorem \ref{main0} guarantees that $n$ is an even perfect number. 
    \item[(b)] If $\beta = 6$, then $2\le \lambda \le 4$ and so $2\le s\le 5$. By \eqref{v4}, $2\le \alpha\le 5$ and by \eqref{k2}, we acquire
       $$p^5\ |\ \frac{2^{5\alpha}-1}{31}\ \le\ \frac{2^{25}-1}{31}.$$
    Hence, $p\in\{3, 7, 11\}$. Computation shows that for each pair $(\alpha, p)$, $\eqref{k2}$ does not hold.
    \end{itemize}
\end{itemize}
We have finished the proof. 
\end{proof}
\appendix
\section{Technical proofs used for Lemma \ref{f}}
We provide proofs of claim(s) made in the proof of Lemma \ref{f}. Notation from Lemma \ref{f} is retained here.  
\begin{lem}\label{vs1}
For all odd $k\ge 3$, we have $2^{k+1}\ ||\ (2^k-1)^{2k}-1$.
\end{lem}
\begin{proof}
Write
\begin{align*}
    (2^k-1)^{2k}-1 \ =\ \sum_{i=0}^{2k} \binom{2k}{i}(2^k)^{2k-i}(-1)^i - 1\ =\ \sum_{i=0}^{2k-1}\binom{2k}{i}(2^k)^{2k-i}(-1)^i.
\end{align*}
When $i = 2k-1$, we have the term $-2k\cdot 2^k = -k2^{k+1}$. Because $k$ is odd, $2^{k+1} \ ||\ k2^{k+1}$. This finishes our proof. 
\end{proof}

\begin{lem}\label{cando} The following holds
$$2^{v+k}\ ||\ (2^k-1)^{\beta k} - 1.$$
\end{lem}
\begin{proof}
We prove by induction on $v$. When $v= 1$, write
\begin{align*}
    (2^k-1)^{\beta k} - 1\ =\ (2^k-1)^{2k\beta_1} - 1\ =\ ((2^k-1)^{2k}-1)\sum_{i=1}^{\beta_1} (2^k-1)^{2k(\beta_1-i)}.
\end{align*}
Because the summation is $1\Mod 2$ and by Lemma \ref{vs1}, $2^{k+1}\ ||\ (2^{k}-1)^{2k}-1$, our claim holds for $v = 1$. Inductive hypothesis: suppose that there exists $z\ge 1$ such that the claim holds for all $1\le v\le z$. We show that it holds for $v = z+1$. We have
\begin{align*}
    (2^k-1)^{2^{z+1}\beta_1 k} - 1\ =\ ((2^k-1)^{2^z\beta_1 k}-1)((2^k-1)^{2^z\beta_1 k}+1).
\end{align*}
By the inductive hypothesis, $2^{z+k}\ ||\ (2^k-1)^{2^z\beta_1 k}-1$, so it suffices to show that $2\ ||\ (2^k-1)^{2^z\beta_1 k}+1$. Observe that 
\begin{align*}
    (2^k-1)^{2^z\beta_1 k}+1 \ =\ (4^k-2^{k+1}+1)^{2^{z-1}\beta_1 k} + 1 \equiv 2\Mod 4.
\end{align*}
Hence, $2\ ||\ (2^k-1)^{2^z\beta_1 k}+1$, as desired. This completes our proof. 
\end{proof}

\begin{lem}\label{appr}
Let $m$ be chosen such that $(2^k-1)^m\ ||\ 2^{(2^k-1)k}-1$. Then for all $u\ge 0$,
$$(2^k-1)^{u+m} \ ||\ 2^{(2^k-1)^{u+1}k\alpha_1}-1.$$
\end{lem}
\begin{proof}
First, we claim that $m\ge 2$. To prove this, write
$$2^{(2^k-1)k}-1 \ =\ (2^k-1)\sum_{i=2}^{2^k}(2^k)^{(2^k-i)}.$$
Since each term in the summation is congruent to $1\Mod 2^k-1$ and there are $2^k-1$ terms, the summation is divisible by $2^k-1$. Therefore, $(2^k-1)^2\ |\ 2^{(2^k-1)k}-1$.

We are ready to prove the lemma. We proceed by induction. For $u = 0$, write 
\begin{align*}2^{(2^k-1)k\alpha_1}-1&\ =\ (2^{(2^k-1)k}-1)(2^{(2^k-1)k(\alpha_1-1)}+2^{(2^k-1)k(\alpha_1-2)}+\cdots + 1)\\
&\ =\ (2^{(2^k-1)k}-1)\sum_{i=1}^{\alpha_1} (2^k)^{(2^k-1)(\alpha_1-i)}.
\end{align*}
By assumption, $(2^k-1)^m\ ||\ 2^{(2^k-1)k}-1$. Each term in the summation $\sum_{i=1}^{\alpha_1} (2^k)^{(2^k-1)(\alpha_1-i)}$ is congruent to $1\Mod 2^k-1$, so the summation is congruent to $\alpha_1\Mod 2^k-1$. Hence, our lemma holds for $u = 0$. Inductive hypothesis: suppose that there exists $z\ge 0$ such that our lemma holds for all $u\le z$. We show that it holds for $u = z+1$. Write 
\begin{align*}
2^{(2^k-1)^{z+2}k\alpha_1}-1&\ =\ (2^{(2^k-1)^{z+1}k\alpha_1}-1)\cdot \\
&(2^{(2^k-1)^{z+1}k\alpha_1(2^k-2)}+2^{(2^k-1)^{z+1}k\alpha_1(2^k-3)}+\cdots +1)\\
&\ =\ (2^{(2^k-1)^{z+1}k\alpha_1}-1)\sum_{i=2}^{2^k} 2^{(2^k-1)^{z+1}k\alpha_1(2^k-i)}.
\end{align*}
By the inductive hypothesis, $(2^k-1)^{z+m}\ ||\  2^{(2^k-1)^{z+1}k\alpha_1}-1$. Each term in the summation is congruent to $1\Mod (2^k-1)^m$. Since there are $2^k-1$ terms, the summation is congruent to $(2^k-1)\Mod (2^k-1)^m$. Because $m\ge 2$, $(2^k-1)$ exactly divides the summation. So, $$(2^{k}-1)^{z+m+1}\mbox{ exactly divides }2^{(2^k-1)^{z+2}k\alpha_1}-1,$$
as desired. This completes our proof.
\end{proof}
\begin{rek}\label{appr2}\normalfont
Note that for all $k\ge 3$, in order that $(2^k-1)^m\le 2^{(2^k-1)k}-1$, we must have $m< 2^k$. By Lemma \ref{appr}, $(2^k-1)^{u+2^k}$ does not divide $2^{(2^k-1)^{u+1}k\alpha_1}-1$ for all $u\ge 0$.
\end{rek}
\section{Technical proofs used for Lemma \ref{not1mod4}}
We provide proofs of claim(s) made in the proof of Lemma \ref{not1mod4}. Notation from Lemma \ref{not1mod4} is retained here. 
\begin{lem}\label{tv} With notation as in Lemma \ref{not1mod4}, the following holds
$$2^{t+v}\ ||\ p^{2^v\beta_1k}-1.$$
\end{lem}
\begin{proof}
We prove by induction on $v$. When $v = 1$, write
\begin{align}p^{2k\beta_1}-1&\ =\ (p^{2k}-1)(p^{2k(\beta_1-1)}+p^{2k(\beta_1-2)} + \cdots + 1)\nonumber\\
&\ =\ (p^k-1)(p^k+1)\sum_{i=1}^{\beta_1}p^{2k(\beta_1-i)}\nonumber\\
&\ =\ (p^k-1)(p+1)\left(\sum_{i=1}^{k} p^{k-i}(-1)^{i+1}\right)\sum_{i=1}^{\beta_1}p^{2k(\beta_1-i)}\label{sl2}.
\end{align}
Since $p+1\equiv 2\Mod 4$, $2\ ||\ (p+1)$. We showed that $2^t||(p^k-1)$ in the proof of Lemma \ref{not1mod4}. Also, the two summations are odd. Therefore, $2^{t+1}\ ||\ p^{2k\beta_1}-1$. 

Inductive hypothesis: suppose that there exists $z\ge 1$ such that our claim holds all $1\le v\le z$. We show that it holds for $v = z+1$. We have
\begin{align*}
p^{2^{z+1}k\beta_1}-1\ =\ p^{(2^zk\beta_1)\cdot 2} - 1\ =\ (p^{2^zk\beta_1} + 1)(p^{2^zk\beta_1} - 1).
\end{align*}
By the inductive hypothesis, $2^{z+t} \ ||\ p^{2^zk\beta_1} - 1$. Also, $p\equiv 1\Mod 4$ implies that $p^{2^zk\beta_1}+1\equiv 2\Mod 4$. So, $2\ ||\ p^{2^zk\beta_1}+1$. Therefore, $2^{z+t+1}\ ||\ p^{2^{z+1}k\beta_1}-1$. We have finished our proof. 
\end{proof}

\section{Technical proofs used for Lemma \ref{not-1mod4}}
We provide proofs of claim(s) made in the proof of Lemma \ref{not-1mod4}. Notation from Lemma \ref{not-1mod4} is retained here. 
\begin{lem}\label{tv2}With notation as in Lemma \ref{not-1mod4}, the following holds
$$2^{v+s-1}\ ||\ p^{k2^v\beta_1}-1.$$
\end{lem}
\begin{proof}
We prove by induction on $v$. When $v = 1$, by \eqref{sl2}, we only consider $(p+1)(p^k-1)$. We showed that $2\ ||\ p^k-1$ in the proof of Lemma \ref{not-1mod4}. Since $2^s\ ||\ (p-1)(p+1)$ and $2\ ||\ p-1$, it follows that $2^{s-1}\ ||\ p+1$. Therefore, $2^s\ ||\ p^{k2\beta_1}-1$.

Inductive hypothesis: suppose that there exists $z\ge 1$ such that for all $1\le v\le z$, our claim holds. We show that it also holds for $v = z+1$. We have
\begin{align*}
p^{2^{z+1}k\beta_1}-1\ =\ p^{(2^zk\beta_1)\cdot 2} - 1\ =\ (p^{2^zk\beta_1} + 1)(p^{2^zk\beta_1} - 1).
\end{align*}
By the inductive hypothesis, $2^{z+s-1}\ ||\ p^{2^zk\beta_1} - 1$. Also, $p^2\equiv 1\Mod 4$ implies that $p^{2^zk\beta_1}+1\equiv 2\Mod 4$. So, $2\ ||\ p^{2^zk\beta_1}+1$.  Therefore, $2^{z+s} \ ||\ p^{2^{z+1}k\beta_1}-1$. We have finished our proof. 
\end{proof}
\section{Technical proofs used for Lemma \ref{3mod4}}
We provide proofs of claim(s) made in the proof of Lemma \ref{3mod4}. Notation from Lemma \ref{3mod4} is retained here. 
\begin{lem}\label{sl3}
With notation as in Lemma \ref{3mod4}, the following holds
$$2^{\lambda+v}\ ||\ (2^\lambda p_1-1)^{2^v\beta_1}-1.$$
\end{lem}
\begin{proof}
We prove by induction on $v$. Observe that 
\begin{align*}
    (2^\lambda p_1-1)^{2\beta_1}-1&\ =\ \sum_{i=0}^{2\beta_1}\binom{2\beta_1}{i}(2^\lambda p_1)^{2\beta_1-i}(-1)^i - 1\\
    &\ =\  \sum_{i=0}^{2\beta_1-1}\binom{2\beta_1}{i}(2^\lambda p_1)^{2\beta_1-i}(-1)^i,
\end{align*}
which clearly indicates that $2^{\lambda+1}\ ||\ (2^\lambda p_1-1)^{2\beta_1}-1$. So, the claim holds for $v = 1$. Inductive hypothesis: suppose that there exists $z\ge 1$ such that for all $1\le v\le z$, the claim holds. We prove that it holds for $v = z+1$. We have
\begin{align*}
    (2^\lambda p_1-1)^{2^{z+1}\beta_1}-1\ =\ ( (2^\lambda p_1-1)^{2^{z}\beta_1}-1)( (2^\lambda p_1-1)^{2^{z}\beta_1}+1).
\end{align*}
By the inductive hypothesis, $2^{\lambda+z} \ ||\ (2^\lambda p_1-1)^{2^{z}\beta_1}-1$. Also, $(2^\lambda p_1-1)^{2^{z}\beta_1}+1\equiv 2\Mod 4$ since $\lambda\ge 2$. Hence, $2^{\lambda+z + 1} \ ||\ (2^\lambda p_1-1)^{2^{z+1}\beta_1}-1$, as desired. 
\end{proof}


\end{document}